\documentclass[11pt]{article}

\usepackage{amsmath, amssymb, amsthm}
\usepackage[utf8]{inputenc}
\usepackage[T1]{fontenc}
\usepackage{geometry}
\usepackage{hyperref}
\usepackage{mathtools}

\newtheorem{theorem}{Theorem}[section]
\newtheorem{lemma}[theorem]{Lemma}
\newtheorem{proposition}[theorem]{Proposition}
\newtheorem{corollary}[theorem]{Corollary}
\theoremstyle{definition}
\newtheorem{definition}[theorem]{Definition}

\theoremstyle{remark}
\newtheorem{remark}[theorem]{Remark}

\newcommand{\C}{\mathbb{C}}

\newcommand{\Z}{\mathbb{Z}}
\newcommand{\D}{\mathbb{D}}
\newcommand{\Acal}{\mathcal{A}}
\newcommand{\Bcal}{\mathcal{B}}
\newcommand{\Mcal}{\mathcal{M}}
\newcommand{\Fcal}{\mathcal{F}}
\newcommand{\fa}{\mathfrak{a}}
\newcommand{\fb}{\mathfrak{b}}
\newcommand{\ff}{\mathfrak{f}}

\title{The Pseudo-Analytic Charge}
\author{Daniel Alay\'on-Solarz\thanks{\texttt{danieldaniel@gmail.com}}}
\date{July 2026}

\begin{document}

\maketitle

\begin{abstract}
The framed Beltrami--Vekua equation
$\Phi(w_{\bar z} - \mu w_z) + \Psi(\overline{w_z} - \mu\overline{w_{\bar z}}) + \mathfrak{a}w + \mathfrak{b}\bar w = \mathfrak{f}$,
with $|\mu|<1$ and $|\Phi|>|\Psi|$, carries a numerator field
$N = \Phi\mathfrak{b} - \Psi\mathfrak{a} - W_L(\Phi,\Psi)$ whose
weighted modulus integrates to the pseudo-analytic mass. This paper
extracts the integer carried by the same field. When the zero set of
$N$ is compactly contained in a bounded simply connected domain, the
winding number of $N$ along any enclosing curve --- the
\emph{pseudo-analytic charge} $n \in \mathbb{Z}$ --- is invariant under
every recombination $w = \varphi w' + \psi\bar w'$ of the unknown,
every scaling of the equation, and every orientation-preserving $C^1$
change of variables: recombinations multiply $N$ by the positive factor
$|\varphi|^2 - |\psi|^2$, so their invariance is exact, while on
multiply connected domains the other two actions fix the component
charges only in $\mathbb{Z}/2\mathbb{Z}$ and the total charge exactly.
The charge is a Brouwer degree: it localizes at the zeros of $N$,
vortices which no action of the class creates or destroys; an isolated
vortex persists under perturbation of the data precisely when its
local charge is non-zero. It
involves the Beltrami coefficient only through the $L$-Wronskian of the
frame, and is $\mu$-independent wherever
$W_\partial(\Phi,\Psi) \equiv 0$ --- in particular at the trivial
frame, where $N = \mathcal{B}$ and the charge is the gauge-invariant
winding of the coefficient of the Beltrami--Vekua equation. Mass and
charge are independent: every pair in $(0,\infty)\times\mathbb{Z}$ is
realized.
\end{abstract}

\section{Introduction}\label{sec:intro}

A first-order real planar elliptic system normalizes, by pointwise
algebra alone, to a \emph{framed Beltrami--Vekua equation}
\begin{equation}\label{eq:framed}
\Phi\,(w_{\bar z} - \mu\, w_z) \;+\; \Psi\,(\overline{w_z} - \mu\,\overline{w_{\bar z}}) \;+\; \fa\, w \;+\; \fb\, \bar w \;=\; \ff,
\qquad |\mu| < 1, \quad |\Phi| > |\Psi|,
\end{equation}
on a domain $\Omega \subset \C$ \cite{framed}. The class \eqref{eq:framed}
is closed under the recombinations of the unknown
$w = \varphi w' + \psi \bar w'$ with $|\varphi| > |\psi|$, under the
scalings of the equation, and under orientation-preserving changes of
variables, with closed transformation laws for all the data. Out of those
laws, when the frame $(\varphi, \psi)$ is smooth, \cite{framed} extracted a single scalar field, the
\emph{numerator field}
\begin{equation}\label{eq:N-def}
N \;:=\; \Phi\,\fb \;-\; \Psi\,\fa \;-\; W_L(\Phi, \Psi),
\qquad W_L(\Phi, \Psi) := \Phi\, L\Psi - \Psi\, L\Phi,
\qquad L := \bar\partial - \mu\,\partial,
\end{equation}
whose modulus, correctly weighted, is an invariant density: the 2-form
$\Theta = |N|^2\,(|\Phi|^2 - |\Psi|^2)^{-2}(1 - |\mu|^2)^{-1}\,dx\,dy$
is unchanged by recombinations and scalings and covariant under changes of
variables, so that the \emph{pseudo-analytic mass}
$\Mcal = \int_\Omega \Theta$ is an invariant of the equivalence class
\cite{mass, framed}.

The mass is a metric invariant: it consumes the modulus of $N$,
weighted by the frame determinant and by the conformal factor of $\mu$,
and nothing of $\arg N$ enters it. The present paper
extracts the topological content of the same field --- the invariant
carried by $\arg N$ alone. We define the \emph{pseudo-analytic charge} of the
equation \eqref{eq:framed} as the winding number
\begin{equation}\label{eq:charge-def}
n \;:=\; \frac{1}{2\pi} \oint_{\partial\Omega} d\,\arg N \;\in\; \Z,
\end{equation}
whenever $N$ is continuous on $\overline\Omega$ and non-vanishing on
$\partial\Omega$, with $\Omega$ bounded and simply connected.

That $n$ is an invariant is a consequence of the exactness of the
transformation laws of $N$ established in \cite{framed}. Under the
recombination $w = \varphi w' + \psi \bar w'$ the numerator field obeys
\begin{equation}\label{eq:N-law-intro}
N' \;=\; \bigl(|\varphi|^2 - |\psi|^2\bigr)\, N,
\end{equation}
a \emph{real positive} multiple: the argument of $N$ does not move at
all, and the charge is invariant on the nose. Under a scaling
$E \mapsto cE$ one has $N' = c^2 N$, and under an orientation-preserving
change of variables $N$ pulls back with the single non-vanishing weight
$\rho^{-1}$; on a simply connected domain neither factor can wind, and
the charge is unchanged. The charge, moreover, needs no weight in the Beltrami coefficient: the
conformal factor $(1 - |\mu|^2)^{-1}$ of the mass density has no
counterpart in \eqref{eq:charge-def}. The coefficient $\mu$ does enter
the field itself, through the $L$-Wronskian
$W_L = W_{\bar\partial} - \mu\, W_{\partial}$, and hence enters the
charge exactly when the frame's $\partial$-Wronskian survives; on the
slices with $W_\partial(\Phi, \Psi) \equiv 0$ --- the trivial frame
among them --- the charge is independent of $\mu$
(Theorem~\ref{thm:invariance}, Remark~\ref{rem:mu-real}). Mass and charge thus read the two coordinates of the field $N$ ---
the mass its weighted modulus, the charge its winding argument ---, and we show by explicit examples that they are
independent: every pair $(\Mcal, n) \in (0, \infty) \times \Z$ is
realized. There is no Bogomolny bound tying the energy to the vortex
number here, in contrast with the Ginzburg--Landau functional whose
vocabulary the charge otherwise borrows.

The name \emph{vortex} is earned in the usual way: the charge localizes.
If $N$ vanishes only on a discrete set, $n$ is the algebraic sum of the
local windings of $N$ around its zeros, and a non-zero charge forces $N$
to vanish somewhere in the interior. A non-zero charge is accordingly an
obstruction to normalizing $N$ real and positive on all of $\Omega$ by
the group actions; the vanishing of the charge is necessary for such a
normalization but not sufficient, the full
account of positive normalization --- and of an analytic obstruction that
can survive at a vortex when the charge vanishes --- being deferred to
Proposition~\ref{prop:normalization} and Remark~\ref{rem:oscillation}.
What the charge controls exactly is the winding of $\arg N$ along an
enclosing curve. One caveat is owed at the outset: $N$ is \emph{data}, not a
solution of an elliptic equation, so no similarity principle governs its
zero set --- the zeros need not be isolated, and no local index need be
positive. The boundary winding \eqref{eq:charge-def} is the only count
that is intrinsically meaningful, and it is the one we take.

The framed formulation buys exactness at a price, and the price is
regularity: the numerator field spends a derivative on the frame ---
the $L$-Wronskian does not exist below $C^1$ (more precisely, below a
$C^1$ frame ratio, per the scaling convention of
Section~\ref{sec:framed}) --- whereas the slice winds the merely
continuous coefficient $\Bcal$ at no differentiability cost. What the
expenditure purchases is the statement becoming clean. 
On the trivial-frame slice
$(\Phi, \Psi) = (1, 0)$ --- the Beltrami--Vekua equation
$w_{\bar z} - \mu w_z + \Acal w + \Bcal \bar w = \Fcal$ of \cite{mass}
--- the numerator field reduces to $N = \Bcal$, and the charge becomes
the boundary winding of the coefficient $\Bcal$. But the substitutions
preserving that slice are the gauges $w = \varphi w'$, under which
$\Bcal \mapsto \Bcal\,\bar\varphi/\varphi$: the argument moves by
$-2\arg\varphi$, and invariance of the winding requires an auxiliary
degree argument, which moreover degenerates to a mod-$2$ statement on
multiply connected domains. In the framed class the positive factor
\eqref{eq:N-law-intro} makes the invariance under the full substitution
group exact and hypothesis-free, and confines whatever topology remains
to the scalings and the changes of variables, where it is transparent.
The charge is also computable \emph{pre-uniformization}: like the mass,
it is read directly off the data $(\mu, \Phi, \Psi, \fa, \fb)$ of any
representative, with no normal form to reach first.

The paper is organized as follows. Section~\ref{sec:framed} recalls the
framed class and the transformation laws of $N$ from \cite{framed}.
Section~\ref{sec:charge} defines the charge and proves its invariance,
with the precise statement on multiply connected domains.
Section~\ref{sec:vortices} localizes the charge at the zeros of $N$ and
identifies it as the obstruction to a positive normalization.
Section~\ref{sec:trivial} restricts to the trivial-frame slice and
recovers the winding of $\Bcal$. Section~\ref{sec:independence} proves
the independence of mass and charge. Section~\ref{sec:discussion} relates
the charge to the index of Vekua's Riemann--Hilbert problem and records
what survives at measurable regularity.

\section{The framed Beltrami--Vekua class}\label{sec:framed}

This section fixes notation and recalls from \cite{framed} the three
group actions on the class \eqref{eq:framed} and the exact transformation
laws of the numerator field. The laws at $C^1$ regularity are proved in
\cite{framed}; new here are only the two reading conventions --- the
definition of the numerator field of a $C^0$-scaled equation, and the
$\rho^{-1}$ reading of the pullback --- which extend those laws to the
uses this paper makes of them..

Throughout, $\Omega \subset \C$ is a bounded domain, simply connected
until Remark~\ref{rem:multiply} says otherwise. The regularity is that of
\cite{framed}: the frame $(\Phi, \Psi)$ and the substitutions are $C^1$;
the Beltrami coefficient $\mu$ is continuous with $|\mu| < 1$ on
$\overline\Omega$; the lower-order data $\fa, \fb, \ff$ are $C^0$. The
Beltrami coefficient $\mu$ is never differentiated: it enters only as a bounded
multiplier inside the pair of first-order operators
\[
L \;:=\; \bar\partial - \mu\,\partial, \qquad
M \;:=\; \partial - \bar\mu\,\bar\partial,
\]
which intertwine through conjugation, $L\bar g = \overline{Mg}$.

Three group actions preserve the class \eqref{eq:framed}.

\emph{Substitutions.} Under the recombination of the unknown
$w = \varphi w' + \psi \bar w'$, with $\varphi, \psi \in C^1(\Omega; \C)$
and $|\varphi| > |\psi|$ pointwise, the data transform by
\begin{equation}\label{eq:sub-law}
(\Phi, \Psi) \;\longmapsto\; \bigl(\Phi\varphi + \Psi\bar\psi,\; \Phi\psi + \Psi\bar\varphi\bigr),
\qquad \mu \;\longmapsto\; \mu,
\end{equation}
with the lower-order data picking up derivatives of the substitution
($\fa \mapsto \Phi\,L\varphi + \Psi\,\overline{M\psi} + \fa\varphi + \fb\bar\psi$
and its companion), and $\ff$ unchanged. The frame determinant is
multiplicative,
\begin{equation}\label{eq:det-mult}
|\Phi'|^2 - |\Psi'|^2 \;=\; \bigl(|\Phi|^2 - |\Psi|^2\bigr)\, D,
\qquad D \;:=\; |\varphi|^2 - |\psi|^2 \;>\; 0,
\end{equation}
so the frame condition $|\Phi| > |\Psi|$ is preserved.

\emph{Scalings.} The equation may be multiplied through:
$E \mapsto cE$ with $c \in C^0(\Omega; \C^*)$, sending
$(\Phi, \Psi, \fa, \fb, \ff) \mapsto c\,(\Phi, \Psi, \fa, \fb, \ff)$ and
fixing $\mu$. For $c \in C^1$ the scaled equation is again a $C^1$-framed
representative and the law \eqref{eq:N-scale} below is a computation.
For merely continuous $c$ the scaled frame leaves the $C^1$ class, and
the numerator field of the scaled equation is \emph{defined} through
any de-scaling: if $cE$ with $E$ $C^1$-framed, set $N_{cE} := c^2 N_E$.
This is well-defined --- if $c_1 E_1 = c_2 E_2$ then
$c_1/c_2 = \Phi_2/\Phi_1$ is a quotient of $C^1$ functions with
zero-free denominator, hence a $C^1$ scaling relating $E_1$ to $E_2$,
and the computed law gives $c_2^2 N_2 = c_1^2 N_1$ --- and it extends
\eqref{eq:N-def} to precisely the continuous-framed equations whose
frame ratio $\Psi/\Phi$ is $C^1$.

\emph{Changes of variables.} If the equation lives on $\Omega'$ and
$F : \Omega \to \Omega'$ is an orientation-preserving $C^1$
diffeomorphism with Jacobian $J = |F_z|^2 - |F_{\bar z}|^2 > 0$, then
$h = w \circ F$ satisfies a framed equation on $\Omega$ over the
pulled-back Beltrami coefficient $\tilde\mu$,
with frame $\rho\,(\Phi \circ F, \Psi \circ F)$ and lower-order data
$(\fa, \fb, \ff) \circ F$, where
\begin{equation}\label{eq:rho-def}
\rho \;:=\; \frac{F_z + (\mu \circ F)\,\overline{F_{\bar z}}}{J} \;\neq\; 0 .
\end{equation}
The weight $\rho$ is non-vanishing on $\Omega$ because
$|\mu \circ F|\,|F_{\bar z}| < |F_z|$; this is the only fact about $\rho$
the present paper uses beyond its formula.

The object of study is the \emph{numerator field} \eqref{eq:N-def},
assembled from the frame, its $L$-derivatives, and the lower-order data.
Its transformation laws are exact and closed:

\begin{proposition}[{\cite{framed}}, extended by the conventions above]\label{prop:N-laws}
Under the three actions above, the numerator field obeys
\begin{align}
\text{substitution:} \qquad & N' \;=\; D\, N, \qquad D = |\varphi|^2 - |\psi|^2 > 0, \label{eq:N-sub}\\
\text{scaling:} \qquad & N' \;=\; c^2\, N, \label{eq:N-scale}\\
\text{change of variables:} \qquad & \tilde N \;=\; \frac{N \circ F}{\rho}, \label{eq:N-diffeo}
\end{align}
where in \eqref{eq:N-diffeo} the numerator field of the pulled-back
equation is read off the $C^1$-framed representative obtained by the
scaling $c = \rho^{-1}$, any two such choices differing by
\eqref{eq:N-scale}.
\end{proposition}

\begin{remark}[The $\rho^2$ ambiguity]\label{rem:rho-square}
The reading convention in \eqref{eq:N-diffeo} conceals an ambiguity,
harmless but worth surfacing. The pulled-back equation arrives with
the frame $\rho\,(\Phi \circ F, \Psi \circ F)$; the field $\tilde N$
of \eqref{eq:N-diffeo} is that of its de-scaling by $\rho^{-1}$, and
any other admissible de-scaling changes $\tilde N$ by the square $c^2$
of a zero-free continuous function. (When $F \in C^2$, so that
$\rho \in C^1$, the raw representative is itself $C^1$-framed, with
numerator field $\rho^2 \tilde N = \rho\,(N \circ F)$.) The numerator
field of the pulled-back equation is thus well defined exactly up to
the ambiguity \eqref{eq:N-scale} --- which is precisely the invariance
the charge will be shown to possess: no effect on a simply connected
domain (Theorem~\ref{thm:invariance}), even shifts of the component
windings on a multiply connected one (Remark~\ref{rem:multiply}).
\end{remark}

The law \eqref{eq:N-sub} is the heart of the matter, and it is worth
recalling why it holds: substituting $w = \varphi w' + \psi \bar w'$
contaminates the lower-order data with derivatives of $(\varphi, \psi)$,
and the $L$-Wronskian $W_L(\Phi, \Psi)$ acquires under \eqref{eq:sub-law}
a transformation defect which is \emph{the same expression}; in the
combination \eqref{eq:N-def} the two cancel identically, leaving the
clean factor $D$. The Wronskian is in the formula so that this can
happen \cite{framed}.

For the mass, the factor $D$ in \eqref{eq:N-sub} is a nuisance to be
cancelled: $|N|^2$ is divided by the squared frame determinant, which by
\eqref{eq:det-mult} absorbs $D^2$, and the conformal weight
$(1 - |\mu|^2)^{-1}$ then makes the resulting 2-form covariant under
\eqref{eq:N-diffeo}. For the charge, the same factor is not a nuisance
but a gift: $D$ is real and positive, so \eqref{eq:N-sub} does not move
$\arg N$ at all. No compensating weight is needed, and $\mu$ will not
appear in anything that follows. The next section takes the argument of
$N$ and winds it.

\section{The charge}\label{sec:charge}

The winding \eqref{eq:charge-def} was stated in the introduction as an
integral over $\partial\Omega$, which presumes a boundary regular enough
to integrate over. No such assumption is needed. Only the germ of $N$
near the boundary matters, and the correct definition winds $N$ along
curves \emph{inside} $\Omega$, where everything in sight is defined and
continuous. For a continuous loop $\sigma : [0,1] \to \C^*$, its winding
number $\operatorname{w}(\sigma) \in \Z$ is $(\theta(1) - \theta(0))/2\pi$
for any continuous argument $\theta$ of $\sigma$; for a continuous
$g : X \to \C^*$ and a loop $\gamma$ in $X$ we write
$\operatorname{w}(g, \gamma) := \operatorname{w}(g \circ \gamma)$. Two
standard facts are used repeatedly: the winding is invariant under free
homotopy of the loop within $X$, and a continuous non-vanishing function
on a simply connected domain admits a continuous logarithm, so its
winding along every loop there vanishes.

\begin{definition}\label{def:charge}
The equation \eqref{eq:framed} on the bounded simply connected domain
$\Omega$ is \emph{admissible} if its numerator field $N$
\eqref{eq:N-def} has compact zero set
\[
Z(N) \;:=\; N^{-1}(0) \;\Subset\; \Omega,
\]
in particular whenever $N$ is zero-free on a neighborhood of
$\partial\Omega$ in $\Omega$. A
positively oriented Jordan curve $\gamma \subset \Omega \setminus Z(N)$
is \emph{enclosing} if $Z(N)$ lies in its interior $U(\gamma)$. The
\emph{pseudo-analytic charge} of an admissible equation is
\begin{equation}\label{eq:charge-curve}
n \;:=\; \operatorname{w}(N, \gamma) \;\in\; \Z
\end{equation}
for any enclosing curve $\gamma$.
\end{definition}

\begin{lemma}[Well-definedness]\label{lem:well-defined}
Enclosing curves exist, and \eqref{eq:charge-curve} does not depend on
the choice.
\end{lemma}

\begin{proof}
\emph{Existence.} Let $\tau : \D \to \Omega$ be a Riemann map. Since
$Z(N)$ is compact in $\Omega$, its preimage $\tau^{-1}(Z(N))$ is compact
in $\D$, hence contained in $\{|\zeta| < r\}$ for some $r < 1$. The
curve $\gamma_r := \tau(\{|\zeta| = r\})$, positively oriented, is an
analytic Jordan curve in $\Omega \setminus Z(N)$ with
$U(\gamma_r) = \tau(\{|\zeta| < r\}) \supset Z(N)$.

\emph{Independence.} First, for $r < s < 1$ both admissible as above,
$\gamma_r$ and $\gamma_s$ are freely homotopic within the image of the
annulus $\{r \le |\zeta| \le s\}$, which avoids $Z(N)$; the windings
agree. Next, let $\gamma$ be an arbitrary enclosing curve. Its trace is
compact in $\Omega$, so for $s$ close to $1$ both $\gamma$ and $Z(N)$
lie in $U(\gamma_s)$. By the Jordan--Schoenflies theorem the compact
region $\overline{U(\gamma_s)} \setminus U(\gamma)$ between the two
curves is a closed topological annulus with boundary
$\gamma \cup \gamma_s$; it avoids $Z(N)$, since $Z(N) \subset U(\gamma)$.
Within it $\gamma$ and $\gamma_s$ are freely homotopic, and both
orientations are positive, so
$\operatorname{w}(N, \gamma) = \operatorname{w}(N, \gamma_s)$.
\end{proof}

\begin{remark}\label{rem:boundary-formula}
When $\Omega$ is a Jordan domain and $N$ extends continuously and
without zeros to $\overline\Omega$, the charge equals the boundary
winding \eqref{eq:charge-def} of the introduction: the annulus argument
of Lemma~\ref{lem:well-defined}, run between an enclosing curve and
$\partial\Omega$ itself, gives the equality. Definition~\ref{def:charge}
simply refuses to let the regularity of $\partial\Omega$ enter a
quantity that never depended on it.
\end{remark}

\begin{theorem}[Invariance of the charge]\label{thm:invariance}
Admissibility and the charge $n$ are invariant under the three group
actions of Section~\ref{sec:framed}:
\begin{itemize}
\item[(i)] under every substitution $w = \varphi w' + \psi \bar w'$ with
$\varphi, \psi \in C^1(\Omega; \C)$, $|\varphi| > |\psi|$ pointwise;
\item[(ii)] under every scaling $E \mapsto cE$ with
$c \in C^0(\Omega; \C^*)$;
\item[(iii)] under every orientation-preserving $C^1$ diffeomorphism
$F : \Omega \to \Omega'$ of bounded simply connected domains.
\end{itemize}
Moreover $\mu$ enters $n$ only through the $L$-Wronskian of the frame,
$W_L = W_{\bar\partial} - \mu\, W_{\partial}$: if
$W_\partial(\Phi, \Psi) \equiv 0$ --- in particular on the
trivial-frame slice --- equations sharing $(\Phi, \Psi, \fa, \fb)$ have
the same numerator field, hence the same charge, for every Beltrami
coefficient. In general the charge is constant along every \emph{uniformly
admissible} deformation $\mu_t$ of the Beltrami coefficient over fixed
$(\Phi, \Psi, \fa, \fb)$: jointly continuous in
$(t, z) \in [0,1] \times \Omega$, with $|\mu_t| < 1$ pointwise for
each $t$, and with $\bigcup_{t \in [0,1]} Z(N_t) \Subset \Omega$.
\end{theorem}

\begin{proof}
(i) By \eqref{eq:N-sub}, $N' = DN$ with $D = |\varphi|^2 - |\psi|^2$
continuous and strictly positive on $\Omega$. Hence $Z(N') = Z(N)$,
admissibility and the family of enclosing curves are unchanged, and
along any enclosing $\gamma$ a continuous argument of $N$ is a
continuous argument of $N'$: the factor $D$ moves the modulus and
nothing else. Thus $n' = n$, with no topology spent.

(ii) By \eqref{eq:N-scale}, $N' = c^2 N$ with $c$ non-vanishing, so
again $Z(N') = Z(N)$ and admissibility is preserved. Along an enclosing
$\gamma$,
\[
\operatorname{w}(N', \gamma)
= \operatorname{w}(N, \gamma) + 2\,\operatorname{w}(c, \gamma)
= \operatorname{w}(N, \gamma),
\]
since $c$ is continuous and zero-free on the simply connected $\Omega$
and therefore has a continuous logarithm there: its winding along every
loop in $\Omega$ vanishes.

(iii) Let $\tilde N = (N \circ F)/\rho$ be the numerator field of the
pulled-back equation, per \eqref{eq:N-diffeo}. Since $F$ is a
homeomorphism, $Z(\tilde N) = F^{-1}(Z(N))$ is compact in $\Omega$ iff
$Z(N)$ is compact in $\Omega'$: admissibility is preserved. Let
$\gamma$ be an enclosing curve for $\tilde N$ in $\Omega$. Then
\[
\operatorname{w}(\tilde N, \gamma)
= \operatorname{w}(N \circ F, \gamma) - \operatorname{w}(\rho, \gamma)
= \operatorname{w}(N, F \circ \gamma),
\]
the $\rho$-term vanishing because $\rho$ \eqref{eq:rho-def} is
continuous and non-vanishing on the simply connected $\Omega$ --- the
same logarithm argument as in (ii); this is the only use of the
ellipticity bound $|\mu| < 1$, which guarantees $\rho \neq 0$. It
remains to see that $F \circ \gamma$ is an enclosing curve for $N$ in
$\Omega'$. It is a Jordan curve, $F$ being injective and continuous; $F$
maps $U(\gamma)$ onto $U(F \circ \gamma)$, so
$Z(N) = F(Z(\tilde N)) \subset F(U(\gamma)) = U(F \circ \gamma)$; and
the orientation is positive because $F$ preserves orientation: for
$p \in U(\gamma)$, the winding of $F \circ \gamma$ about $F(p)$ is the
Brouwer degree $\deg(F, U(\gamma), F(p)) = +1$. Hence
$\operatorname{w}(N, F \circ \gamma) = n'$, and $n = n'$.

Finally, expanding the $L$-Wronskian gives
\[
N \;=\; \bigl(\Phi\fb - \Psi\fa - W_{\bar\partial}(\Phi, \Psi)\bigr)
\;+\; \mu\, W_{\partial}(\Phi, \Psi):
\]
the numerator field is affine in $\mu$ with coefficient $W_\partial$.
If $W_\partial \equiv 0$ the field, and with it the charge, does not
see $\mu$ at all. For the deformation statement, let $\mu_t$ be
uniformly admissible and set
$K := \overline{\bigcup_t Z(N_t)} \Subset \Omega$. Fix a Riemann map
$\tau : \D \to \Omega$ as in Lemma~\ref{lem:well-defined} and choose
$r < 1$ with $K \subset \tau(\{|\zeta| < r\})$: then
$\gamma := \gamma_r$ is a common enclosing curve for every $N_t$. On
the compact set $[0,1] \times \gamma$ the map
$(t, z) \mapsto N_t(z)$ is continuous --- $\mu_t$ enters only as a
bounded multiplier of the fixed continuous field $W_\partial$ --- and
zero-free, so $t \mapsto \operatorname{w}(N_t, \gamma)$ is an
integer-valued continuous function, hence constant.
\end{proof}

\begin{remark}[The $\mu$-dependence is real]\label{rem:mu-real}
Neither hypothesis in the last part of Theorem~\ref{thm:invariance} can
be dropped. On $\D$ take $\Phi = 1$, $\Psi = z/2$, $\fa = \fb = 0$:
then $W_{\bar\partial} = 0$, $W_\partial = 1/2$, and $N = \mu/2$ is the
Beltrami coefficient itself. The coefficients $\mu = 1/2$ and
$\mu = z/2$ give admissible equations of charges $0$ and $1$ over the
same $(\Phi, \Psi, \fa, \fb)$; and $\mu = z/2$, $\mu = \bar z/2$ give
the same vortex set $\{0\}$ with charges $+1$ and $-1$. Uniform
admissibility is likewise necessary: over the same data, a continuous
path of coefficients can carry the vortex out through the boundary and
readmit its reflection --- joining $\varepsilon z$ to
$\varepsilon\bar z$ through $\varepsilon(z - c_t)$,
$\varepsilon((1-s)(z-1) + s(\bar z - 1))$,
$\varepsilon(\bar z - \bar c_t)$, with $\varepsilon < 1/2$, so that
$|\mu_t| \le 2\varepsilon < 1$ on $\overline\D$ throughout --- 
keeping each equation admissible while the
charge jumps from $+1$ to $-1$.
\end{remark}

\begin{remark}[Multiply connected domains]\label{rem:multiply}
Let $\Omega$ have finitely many holes, with $N$ admissible in the same
sense ($Z(N) \Subset \Omega$). Each hole $H_j$, together with the part
of $Z(N)$ assigned to it by a system of separating curves, carries its
own winding $n_j$, and the three actions treat the tuple $(n_j)$
differently. Substitutions still act through the positive factor $D$ of
\eqref{eq:N-sub} and preserve every $n_j$ exactly. Scalings and
changes of variables, however, involve $c^2$ and $\rho^{-1}$, and on a
multiply connected domain a zero-free continuous function may wind
around a hole. Under a scaling each $n_j$ shifts by the even integer
$2\operatorname{w}(c, \gamma_j)$. Under a change of variables the
action is not a shift. Let $\epsilon_j = +1$ or $-1$ according as the
Jordan curve $F \circ \gamma_j$ is positively or negatively oriented
--- the latter occurs exactly when $F$ everts the component,
exchanging the enclosed complementary component with the unbounded
one, as the inversion of an annulus does --- and write $n'_j$ for the
winding of the image field along $F \circ \gamma_j$ taken with
positive orientation, so that
$\operatorname{w}(N, F \circ \gamma_j) = \epsilon_j\, n'_j$. The
homotopy $t \mapsto \bigl(F_z + t\,(\mu \circ F)\,\overline{F_{\bar z}}\bigr)/J$,
$t \in [0,1]$, is zero-free by the ellipticity bound, so
$\operatorname{w}(\rho, \gamma_j) = \operatorname{w}(F_z, \gamma_j)$;
taking $\gamma_j$ a regular $C^1$ Jordan curve, as we may, the tangent of the
image curve is $(F \circ \gamma_j)' = F_z\,\gamma_j'\,(1 + q)$ with
$q = (F_{\bar z}/F_z)\,\overline{\gamma_j'}/\gamma_j'$ of modulus
$< 1$, and the Umlaufsatz applied to the two Jordan curves $\gamma_j$
and $F \circ \gamma_j$ gives
$\operatorname{w}(F_z, \gamma_j) = \epsilon_j - 1 \in \{0, -2\}$. The
transformation law is therefore
\[
\tilde n_j
\;=\; \operatorname{w}(N, F \circ \gamma_j)
      - \operatorname{w}(\rho, \gamma_j)
\;=\; \epsilon_j\, n'_j - (\epsilon_j - 1)
\;=\;
\begin{cases}
n'_j, & \epsilon_j = +1,\\[2pt]
2 - n'_j, & \epsilon_j = -1:
\end{cases}
\]
exact on components whose boundary orientation $F$ preserves, the
affine negation $n \mapsto 2 - n$ on everted ones, and in either case
$\tilde n_j \equiv n'_j \pmod 2$. The component charges are therefore
invariants only in $\Z / 2\Z$. What survives
exactly is the \emph{total} charge, the winding along the full oriented
boundary cycle (outer curve minus hole curves, each realized by nearby
interior curves): a continuous zero-free function on $\Omega$ has total
winding zero along such a cycle, since the cycle bounds in
$\Omega \setminus Z(c) = \Omega$, and the contributions of $c^2$ and
$\rho^{-1}$ cancel in the sum, while $F$, being an orientation-preserving
homeomorphism, carries the boundary cycle of $\Omega$ to a cycle
homologous in $\Omega' \setminus Z(N)$ to the boundary cycle of
$\Omega'$: in $H_1(\Omega \setminus Z(\tilde N))$ the boundary cycle is
the sum of small positively oriented cycles about the vortex clusters,
and their images remain positively oriented by the degree argument of
Theorem~\ref{thm:invariance}(iii), which applies because the small
disks --- unlike the hole curves --- lie inside $\Omega$.
The clean dichotomy --- component
charges mod $2$, total charge in $\Z$ --- is another instance of the
framed formulation earning its keep: the substitution group, the only
one that mixes $w$ with $\bar w$, is precisely the one that costs
nothing.
\end{remark}

\section{Localization and vortices}\label{sec:vortices}

The charge was defined as a single winding along a curve near the
boundary. This section shows it is assembled from local contributions
sitting at the zeros of $N$ --- the \emph{vortices} of the equation ---
and settles what the charge does and does not control about them. The
right instrument is the Brouwer degree, whose planar theory we use in
its standard form \cite{deimling}: for $V \subset \C$ bounded open and
$g : \overline V \to \C$ continuous with $g \neq 0$ on $\partial V$,
the degree $\deg(g, V, 0) \in \Z$ is defined, depends only on
$g|_{\partial V}$ up to zero-free homotopy, is additive over disjoint
open subsets containing all zeros, and is non-zero only if $g$ vanishes
somewhere in $V$; when $V$ is the interior of a Jordan curve $\gamma$,
it equals the winding $\operatorname{w}(g, \gamma)$.

\begin{proposition}[The charge is a degree]\label{prop:degree}
Let the equation be admissible and $\gamma$ an enclosing curve. Then
\begin{equation}\label{eq:charge-degree}
n \;=\; \deg\bigl(N,\, U(\gamma),\, 0\bigr).
\end{equation}
Consequently:
\begin{itemize}
\item[(i)] if $Z(N) \subset V_1 \cup \dots \cup V_k$ with
$V_i \Subset \Omega$ open and pairwise disjoint, then
$n = \sum_i \deg(N, V_i, 0)$;
\item[(ii)] if $n \neq 0$ then $Z(N) \neq \emptyset$: a non-zero charge
forces a vortex;
\item[(iii)] if $Z(N)$ is finite, then $n = \sum_{p \in Z(N)} n_p$,
where the \emph{local charge} $n_p := \operatorname{w}(N, \partial D_p)$
is the winding along the boundary of any small disk
$D_p \ni p$ isolating $p$.
\end{itemize}
\end{proposition}

\begin{proof}
\eqref{eq:charge-degree} is the degree--winding identity above;
(i) is additivity and excision; (ii) is the solution property of the
degree; (iii) is (i) with disks.
\end{proof}

Before asking what the charge says about the vortices, one should ask
what is invariant about the vortices themselves. The answer is:
everything, as a set.

\begin{proposition}[Invariance of the vortex set]\label{prop:vortex-set}
Under substitutions and scalings, $Z(N') = Z(N)$; under an
orientation-preserving change of variables $F$,
$Z(\tilde N) = F^{-1}(Z(N))$. The vortex set is an invariant of the
equivalence class, as a subset of $\Omega$ up to the homeomorphisms of
the class itself.
\end{proposition}

\begin{proof}
Immediate from \eqref{eq:N-sub}--\eqref{eq:N-diffeo}: the factors $D$,
$c^2$, $\rho^{-1}$ are zero-free.
\end{proof}

The group actions therefore never create or destroy a vortex. What can
destroy a vortex is a \emph{perturbation of the data} --- and this is
exactly the distinction the charge governs. The numerator field depends
continuously on the data: $N$ is a polynomial expression in $\Phi$,
$\Psi$, $L\Phi$, $L\Psi$, $\fa$, $\fb$, so a $C^1$-small perturbation of
the frame together with a $C^0$-small perturbation of $\mu, \fa, \fb$
moves $N$ uniformly little on compact sets. The stability statement is
Rouch\'e's, in degree form:

\begin{proposition}[Stability]\label{prop:stability}
Let $N$, $\tilde N$ be the numerator fields of two admissible equations
and $\gamma$ a common enclosing curve. If
\[
|\tilde N - N| \;<\; |N| \qquad \text{on } \gamma,
\]
then the charges agree, $\tilde n = n$. Locally: if
$\overline D \subset \Omega$ is a closed disk with $N \neq 0$ on
$\partial D$ and $\deg(N, D, 0) \neq 0$, then every continuous field
uniformly closer to $N$ on $\partial D$ than $\min_{\partial D} |N|$
also vanishes in $D$.
\end{proposition}

\begin{proof}
The segment $N_t = (1-t)N + t\tilde N$ is zero-free on $\gamma$
(resp.\ $\partial D$) for $t \in [0,1]$, so
$t \mapsto \operatorname{w}(N_t, \gamma)$ is an integer-valued
continuous function, hence constant; the local statement adds the
solution property of the degree.
\end{proof}

\begin{remark}[Removability at local charge zero]\label{rem:removable}
The converse of the local statement holds and completes the dichotomy.
Let $p$ be an isolated vortex with $n_p = 0$ and $\varepsilon > 0$.
Choose a disk $D_\delta \ni p$ with $|N| \le \varepsilon$ on
$\overline{D_\delta}$ (continuity, $N(p) = 0$) and $N \neq 0$ on
$\partial D_\delta$. Since $n_p = 0$, on $\partial D_\delta$ one may
write $N = e^{u + i\theta}$ with $u, \theta$ real continuous;
Tietze-extend $u$ into $D_\delta$ with values in
$[\min_{\partial D_\delta} u, \max_{\partial D_\delta} u]$, extend
$\theta$ continuously, and replace $N$ on $D_\delta$ by
$e^{\tilde u + i\tilde\theta}$: the result is continuous, zero-free on
$D_\delta$, agrees with $N$ on $\partial D_\delta$, and differs from
$N$ by at most $2\varepsilon$. The perturbed field is again a
numerator field --- and not merely of \emph{some} equation
(Lemma~\ref{lem:realization}) but of a small perturbation of the given
one: $N$ is affine in $\fb$ with zero-free coefficient $\Phi$, so the
modification $\delta N$, supported in $\overline{D_\delta}$ and of
uniform size at most $2\varepsilon$, is realized over the same
$(\mu, \Phi, \Psi, \fa)$ by the perturbation
$\fb \mapsto \fb + \delta N / \Phi$, which is $C^0$-small because
$|\Phi|$ is bounded below on the compact $\overline{D_\delta}$. An isolated vortex thus
persists under all sufficiently small perturbations of the data
precisely when its local charge is non-zero.
\end{remark}

The two propositions together assign the charge its exact job. The
vortex set is an invariant of the class but not of the data: a zero of
local charge $0$ is removable by an arbitrarily small perturbation,
while a zero (or cluster of zeros) of non-zero local charge persists
under every sufficiently small one. The charge --- total or local ---
is the perturbation-proof part of the vortex structure.

It is worth recording how unconstrained that structure is. Solutions of
a homogeneous Vekua equation obey the similarity principle
$w = e^s h$ with $h$ holomorphic \cite{bers, vekua}: their zeros are
isolated and their local windings positive. No such principle governs
$N$, which is data, not a solution --- and in the framed class this is
not merely a failure of proof but a theorem, by free realization:

\begin{lemma}[Realization]\label{lem:realization}
Every $N \in C^0(\Omega; \C)$ is the numerator field of a framed
equation --- indeed of a trivially framed one: take
$(\Phi, \Psi) = (1, 0)$, $\fa = 0$, $\fb = N$, and $\mu$ arbitrary.
Then $W_L(1, 0) = 0$ and the field \eqref{eq:N-def} is $N$ itself.
\end{lemma}

So the vortex set can be any compact subset of $\Omega$, and local
charges, where defined, can take any integer values, negative included.
The charge does not know that $N$ came from an elliptic equation; it
only knows that it transforms by \eqref{eq:N-sub}--\eqref{eq:N-diffeo}.

Finally, the normalization question: when can the vortexless part of
the theory be trivialized? For zero-free $N$ the answer is complete,
and gives the class a canonical form.

\begin{proposition}[Positive normalization]\label{prop:normalization}
An admissible equation is equivalent, by a scaling alone, to one with
$N \equiv 1$ if and only if $N$ is zero-free on $\Omega$. In that case
$n = 0$.
\end{proposition}

\begin{proof}
Necessity is Proposition~\ref{prop:vortex-set}: scalings do not change
$Z(N)$, and $N \equiv 1$ has none. For sufficiency, $N$ zero-free and
continuous on the simply connected $\Omega$ admits a continuous
logarithm, $N = e^g$; the scaling $c := e^{-g/2} \in C^0(\Omega; \C^*)$
gives, by \eqref{eq:N-scale} --- read, for continuous $c$, as the
definition of the scaled numerator field per the convention of
Section~\ref{sec:framed} --- $N' = c^2 N = 1$. The last claim is
Proposition~\ref{prop:degree}(ii).
\end{proof}

\begin{remark}[Scalings suffice]\label{rem:scalings-suffice}
Proposition~\ref{prop:normalization} normalizes by a scaling alone, but
nothing is lost thereby: of the three actions, substitutions multiply
$N$ by the positive factor $D$ and cannot repair a sign or remove a
zero, and changes of variables only relocate $Z(N)$; the obstruction to
normalization by the full group is therefore the same as the
obstruction to normalization by scalings.
\end{remark}

\begin{remark}[The charge is not the whole obstruction]\label{rem:oscillation}
In the presence of vortices, the vanishing of the total charge --- and
even of every local charge --- does not suffice to normalize $N$
non-negative by a continuous scaling. By Lemma~\ref{lem:realization}
realize, near an interior point $p$,
\[
N(z) \;=\; |z - p|\; e^{\,i/|z - p|}, \qquad N(p) = 0,
\]
a continuous field with a single vortex of local charge $0$: the
argument is constant on circles about $p$, so every winding vanishes.
Yet any scaling $c$ with $c^2 N \ge 0$ off $p$ must have
$\arg c \equiv -\tfrac{1}{2}|z-p|^{-1} \pmod{\pi}$ there, which
oscillates without limit as $z \to p$: no such $c$ is continuous at
$p$. The charge captures the topological obstruction to positivity,
and Proposition~\ref{prop:normalization} shows that off the vortex set
it is the only one; at a vortex an analytic obstruction --- the
oscillation of $\arg N$ --- can survive even when all charges vanish.
\end{remark}

The trivial frame used in Lemma~\ref{lem:realization} is where the
present theory meets the Beltrami--Vekua class of \cite{mass}, and the
next section works out that slice in full.

\section{The trivial-frame slice}\label{sec:trivial}

The Beltrami--Vekua class of \cite{mass} is the slice
$(\Phi, \Psi) = (1, 0)$ of the framed class:
\begin{equation}\label{eq:bv}
w_{\bar z} - \mu\, w_z + \Acal\, w + \Bcal\, \bar w \;=\; \Fcal,
\qquad |\mu| < 1 .
\end{equation}
There $W_L(1, 0) = 0$ and the numerator field collapses to the
coefficient itself, $N = \Bcal$: the pseudo-analytic charge of
\eqref{eq:bv} is the winding of $\Bcal$,
\begin{equation}\label{eq:charge-B}
n \;=\; \operatorname{w}(\Bcal, \gamma),
\end{equation}
along any enclosing curve for $Z(\Bcal)$, and its mass density is
$|\Bcal|^2 (1 - |\mu|^2)^{-1}\, dx\, dy$ \cite{mass}. Modulus and
argument of the single coefficient $\Bcal$ thus split between the two
invariants: the mass integrates $|\Bcal|^2$ against the conformal
weight of $\mu$, the charge winds $\arg \Bcal$ and forgets $\mu$
altogether.

To read the invariance of \eqref{eq:charge-B} inside the slice one must
first know its stabilizer.

\begin{proposition}[The gauge group of the slice]\label{prop:gauge}
A substitution--scaling pair preserves the trivial frame if and only if
it is a \emph{gauge}: the substitution $w = \varphi w'$ with
$\varphi \in C^1(\Omega; \C^*)$, $\psi = 0$, followed by the scaling
$c = \varphi^{-1}$. Its action on the data of \eqref{eq:bv} is
\begin{equation}\label{eq:gauge-law}
\Acal \;\longmapsto\; \Acal + \frac{L\varphi}{\varphi}, \qquad
\Bcal \;\longmapsto\; \Bcal\, \frac{\bar\varphi}{\varphi}, \qquad
\Fcal \;\longmapsto\; \frac{\Fcal}{\varphi}, \qquad
\mu \;\longmapsto\; \mu .
\end{equation}
\end{proposition}

\begin{proof}
At the trivial frame the substitution law \eqref{eq:sub-law} gives
$(\Phi', \Psi') = (\varphi, \psi)$, and a subsequent scaling by $c$
gives $(c\varphi, c\psi)$; this equals $(1, 0)$ iff $\psi = 0$ and
$c = \varphi^{-1}$. The condition $|\varphi| > |\psi| = 0$ is
$\varphi \neq 0$. The laws \eqref{eq:gauge-law} are the composite of the
lower-order substitution laws of Section~\ref{sec:framed} at
$(\Phi, \Psi) = (1, 0)$, $\psi = 0$ --- namely
$\Acal \mapsto L\varphi + \Acal\varphi$,
$\Bcal \mapsto \Bcal\bar\varphi$ --- with the division by $\varphi$.
\end{proof}

The gauge factor on $\Bcal$ is $\bar\varphi/\varphi$: unimodular, so
$|\Bcal|$ is pointwise gauge-invariant (whence the mass density), but
\emph{not} positive --- the argument moves by $-2\arg\varphi$. This is
the structural difference between the slice and the framed class, and
it is worth displaying exactly where the invariance of the charge now
draws on topology.

\begin{corollary}\label{cor:slice-charge}
For admissible \eqref{eq:bv} on a bounded simply connected $\Omega$, the
winding \eqref{eq:charge-B} is invariant under every gauge
\eqref{eq:gauge-law} and under every orientation-preserving $C^1$ change
of variables, and coincides with the framed charge of
Definition~\ref{def:charge}. Under a gauge,
\[
\operatorname{w}(\Bcal', \gamma)
\;=\; \operatorname{w}(\Bcal, \gamma) - 2\,\operatorname{w}(\varphi, \gamma)
\;=\; \operatorname{w}(\Bcal, \gamma),
\]
the second equality because the zero-free continuous $\varphi$ on the
simply connected $\Omega$ cannot wind.
\end{corollary}

\begin{proof}
The coincidence with the framed charge is definitional, $N = \Bcal$ on
the slice; the invariances are then Theorem~\ref{thm:invariance}. The
displayed computation is the slice-internal proof: $Z(\Bcal)$ is
unchanged since $\bar\varphi/\varphi \neq 0$, and the winding of
$\bar\varphi/\varphi$ along $\gamma$ is $-2\operatorname{w}(\varphi, \gamma) = 0$.
\end{proof}

The real-gauge normalization of Section~\ref{sec:vortices} also has a
slice-internal form. Note the regularity: gauges are $C^1$, since
$L\varphi$ enters \eqref{eq:gauge-law}.

\begin{proposition}[Real gauge]\label{prop:real-gauge}
If $\Bcal \in C^1(\Omega; \C)$ is zero-free on the simply connected
$\Omega$, there is a gauge carrying \eqref{eq:bv} to an equation with
$\Bcal' = |\Bcal| > 0$. In particular $n = 0$; and conversely no
equation with $n \neq 0$, or with a vortex, admits a gauge making
$\Bcal'$ positive.
\end{proposition}

\begin{proof}
Zero-free $\Bcal \in C^1$ on simply connected $\Omega$ has a $C^1$
logarithm, $\Bcal = e^{u + i\theta}$ with $u, \theta \in C^1$ real. The
gauge $\varphi := e^{i\theta/2}$ gives
$\Bcal' = \Bcal\, e^{-i\theta} = e^u = |\Bcal|$. The converse: a gauge
neither moves $Z(\Bcal)$ nor, by Corollary~\ref{cor:slice-charge}, the
charge, and a positive $\Bcal'$ has empty vortex set and charge $0$.
\end{proof}

\begin{remark}[The charge is pseudo-analytic]\label{rem:pa}
The title of this paper carries a claim that can now be cashed out. On
the slice, $N = \Bcal$ is the coefficient of $\bar w$ --- the term that
makes solutions of \eqref{eq:bv} pseudo-analytic in the sense of
\cite{bers}, rather than gauged solutions of the principal part alone.
Where $\Bcal$ vanishes identically on an open set, the conjugate
coupling can there be removed entirely: gauging away $\Acal$ as well
amounts to solving $L\varphi = -\Acal\,\varphi$ with $\varphi$
zero-free, which for H\"older data $\mu, \Acal$ has local $C^{1}$
solutions $\varphi = e^s$, $Ls = -\Acal$: a principal homeomorphism of
the Beltrami equation, $C^{1,\alpha}$ for H\"older $\mu$, straightens
$L$ to $\bar\partial$, and the Pompeiu integral then solves the
straightened equation \cite{vekua, bers}; at bare $C^0$ data
the solution $s$ is only
H\"older, and the reduction exits the $C^1$ gauge group --- a
regularity delicacy, not a topological one. Modulo that caveat,
$\Omega \setminus Z(N)$ is the locus where the equation is
\emph{irreducibly} pseudo-analytic --- where no gauge uncouples $w$
from $\bar w$ --- and by Proposition~\ref{prop:vortex-set} that locus
is an invariant of the equivalence class. Admissibility then reads: the
equation must be irreducibly pseudo-analytic near $\partial\Omega$.
Equations reducible to the principal part are not of charge zero; they
are \emph{chargeless}, outside the definition altogether. And a
non-zero charge asserts two things at once: the conjugate coupling is
globally irremovable, and it is topologically forced to degenerate ---
to vanish at vortices --- in the interior, the local charge $n_p$
measuring how the coupling twists around the point of its own
degeneration. Pseudo-analyticity near the boundary is what the charge
needs to exist; a non-zero charge is the topological certificate that
the pseudo-analytic character can be neither removed nor kept
non-degenerate.
\end{remark}

\begin{remark}[What the framing buys]\label{rem:mod2}
On a simply connected domain the slice and the framed class tell the
same story, as they must. The difference surfaces on multiply connected
domains, where Remark~\ref{rem:multiply} split the invariance: exact
for substitutions, mod $2$ per hole for scalings and changes of
variables. The slice cannot see
that split. Its stabilizer \eqref{eq:gauge-law} is a substitution
\emph{welded to} a scaling --- the factor
$\bar\varphi/\varphi = |\varphi|^2 \cdot \varphi^{-2}$ is precisely the
positive substitution factor $D = |\varphi|^2$ times the scaling factor
$c^2 = \varphi^{-2}$ --- so every gauge carries the winding-capable
part, and around a hole the component charges of $\Bcal$ are invariants
only in $\Z / 2\Z$. The framed class unwelds the two: the substitutions,
the only actions that mix $w$ with $\bar w$, contribute the harmless
positive factor, and the mod-$2$ loss is confined to the scalings and the
changes of variables, where it is transparent. Framing does not add invariance; it exhibits
where the invariance was always sitting.
\end{remark}

With the slice understood, the two invariants of \eqref{eq:bv} are the
mass, fed by $|\Bcal|$ and weighted by $\mu$, and the charge, wound from
$\arg \Bcal$ and blind to $\mu$. The next section shows they are as
independent as this division of labor suggests.

\section{Independence from the mass}\label{sec:independence}

Mass and charge are both invariants of the equivalence class; the
question is whether either constrains the other. The vocabulary of
Section~\ref{sec:vortices} suggests a comparison with the
Ginzburg--Landau functional, where the answer is famously negative in
the other direction: there, energy is bounded \emph{below} by the vortex
number, through the Bogomolny decomposition of the energy into a perfect
square plus a topological term \cite{jt}. Here no such bound exists, in
either direction: the joint invariant $(\Mcal, n)$ fills its whole
range.

\begin{theorem}[Independence]\label{thm:independence}
For every pair $(m, k) \in (0, \infty) \times \Z$ there is an admissible
equation on $\D$ with $\Mcal = m$ and $n = k$. Moreover, within each
charge sector the mass is unconstrained:
\[
\inf \{\, \Mcal : n = k \,\} \;=\; 0
\quad\text{and}\quad
\sup \{\, \Mcal : n = k \,\} \;=\; \infty
\qquad \text{for every } k \in \Z .
\]
\end{theorem}

\begin{proof}
By Lemma~\ref{lem:realization} it suffices to produce continuous fields
$N = \Bcal$ on $\D$, taking the trivially framed equation with
$\mu = 0$, $\fa = 0$, $\fb = \Bcal$, for which the mass is
$\Mcal = \int_\D |\Bcal|^2\, dx\, dy$ and the charge is the winding of
$\Bcal$. For $k \in \Z$ and $t > 0$ set
\[
\Bcal_{t,k}(z) \;:=\;
\begin{cases}
t\, z^{k}, & k \ge 0,\\[2pt]
t\, \bar z^{\,|k|}, & k < 0 .
\end{cases}
\]
For $k \neq 0$ the vortex set is $\{0\}$ and the charge is the winding
of $z^k$ (resp.\ $\bar z^{|k|}$) on any circle about the origin, namely
$k$; for $k = 0$ the field is zero-free of charge $0$. The mass is
\[
\Mcal(t, k) \;=\; t^2 \int_\D |z|^{2|k|}\, dx\, dy
\;=\; \frac{\pi\, t^2}{|k| + 1}\,,
\]
a bijection of $t \in (0, \infty)$ onto $(0, \infty)$ for each fixed
$k$: choose $t = \sqrt{m(|k|+1)/\pi}$. Letting $t \to 0$ and
$t \to \infty$ gives the infimum and supremum.
\end{proof}

\begin{corollary}[No Bogomolny bound]\label{cor:no-bogomolny}
There is no function $f : \Z \to [0, \infty)$, positive at some
$k \neq 0$, with $\Mcal \ge f(n)$ for all admissible equations; nor any
bound in the reverse direction.
\end{corollary}

\begin{remark}[Why the bound fails]\label{rem:why-no-bound}
The failure is structural, and one line of scaling exposes it: the
family $\Bcal \mapsto t\,\Bcal$ leaves the vortex set and every winding
untouched while carrying the mass continuously to $0$. The
Ginzburg--Landau energy resists this scaling because it penalizes the
\emph{gradient} of the order parameter: a vortex of degree $k$ forces
the phase to turn $k$ times around it, the angular derivative cannot be
scaled away, and the Bogomolny rearrangement converts exactly that
kinetic cost into the topological term $2\pi |k|$ \cite{jt}. The
pseudo-analytic mass contains no derivative of $N$: it is a zeroth-order
$L^2$ quantity of the field itself, the derivatives having already been
spent inside $N$ --- the $L$-Wronskian is part of the field, not of the
density. Winding is therefore free of charge, so to speak: the phase of
$N$ can turn arbitrarily often at arbitrarily small mass. Mass and
charge are not two readings of one energy; they are the metric and
topological shadows of the same field, cast in independent directions.
\end{remark}

\begin{remark}\label{rem:not-complete}
Independence also delimits what the pair $(\Mcal, n)$ can do: its
fibers are enormous, and it separates nothing beyond what it measures.
Mass and charge are invariants of the class, not coordinates on it. The
Beltrami coefficient, entering $n$ only through the frame's
$L$-Wronskian and $\Mcal$ only through the conformal weight, carries
its own, separate invariant content, for which see \cite{mass, framed}.
\end{remark}

\section{Discussion}\label{sec:discussion}

\emph{A discrete invariant.} Mass and charge are invariants of opposite
temperament. The mass varies continuously with the data; the charge, by
Proposition~\ref{prop:stability}, is locally constant: it cannot move
under a perturbation of $(\Phi, \Psi)$ small in $C^1$ and of
$(\mu, \fa, \fb)$ small in $C^0$ near an enclosing curve, so it labels,
within the admissible equations, something like connected components ---
a topological quantum number of the equation. The pairing is the
familiar one from gauge theory: a continuous modulus and a discrete
charge, attached here to a single field $N$. One sector is excluded by
fiat: equations with $N \equiv 0$ --- among them the pure Beltrami
equations with trivial lower-order part, the zero-mass locus of
\cite{mass, framed} --- are not admissible, and carry no charge. The
charge is an invariant of the massive sector --- equivalently, by
Remark~\ref{rem:pa}, of the sector irreducibly pseudo-analytic near the
boundary.

\emph{Vekua's index.} The theory of \cite{vekua} attaches an integer to
a different winding: for the Riemann--Hilbert problem
$\operatorname{Re}(\bar\lambda w) = \gamma$ on $\partial\Omega$, with
$|\lambda| = 1$, the index of the \emph{boundary condition} --- the
winding of $\lambda$ --- controls the dimension counts of solvability.
That index is external: it belongs to the datum $\lambda$ on the
boundary, and can be prescribed freely over a fixed equation. The
pseudo-analytic charge is internal: it belongs to the coefficients, is
pinned by the equivalence class, and exists before any boundary problem
is posed. The two windings live on the same circle, and it would be
surprising if they never spoke: one expects the solvability theory of
boundary value problems for admissible equations to see both integers
--- and, more finely, to see the vortices of $N$, near which the
equation degenerates in a gauge-irremovable way
(Proposition~\ref{prop:vortex-set}). We leave the interaction of $n$
with Vekua's index, and the local theory of solutions at a vortex of
prescribed local charge, as open problems.

\emph{Multiply connected domains, sharply.} The mod-$2$ loss of
Remark~\ref{rem:multiply} is not an artifact of proof. On the annulus
$\{1 < |z| < 2\}$ the scaling $c(z) = z$ is continuous and zero-free,
and shifts the component charge around the hole by
$2\operatorname{w}(z, \gamma) = 2$; the inversion $F(z) = 1/z$, from
$\{1 < |z| < 2\}$ onto $\{1/2 < |w| < 1\}$, everts the hole and
realizes the negation $n \mapsto 2 - n$ of Remark~\ref{rem:multiply}:
pulling back the field $N \equiv 1$ (with $\mu = 0$, so that
$\rho = -\bar z^{\,2}$ and $\tilde N = -\bar z^{\,-2}$) carries
component charge $0$ to $2$, and pulling back $N(w) = \bar w^{\,-2}$
(so that $\tilde N = \bar z^{\,2}/(-\bar z^{\,2}) = -1$) carries $2$
to $0$. The component charges genuinely live in $\Z / 2\Z$, while the
total charge, and the component charges under substitutions alone, are
exact. Nothing sharper is available.

\emph{Measurable regularity.} The companion paper \cite{framed} pushes
the mass to measurable data: $\mu$ measurable and locally elliptic, the
frame in $W^{1,2}_{\mathrm{loc}} \cap L^\infty_{\mathrm{loc}}$, the
changes of variables quasiconformal. The mass survives there because it
is an integral: $N$ lands in $L^2_{\mathrm{loc}}$ and $\Theta$ asks for
no more. The charge asks for more. A winding needs trace data on
curves, and an $L^2_{\mathrm{loc}}$ field has none; at full measurable
regularity the charge, as defined, does not exist. Two regimes below
full generality do support it. If the data are measurable in the bulk
but continuous on a collar of the boundary, Definition~\ref{def:charge}
applies verbatim, enclosing curves being creatures of the collar. And
if $N$ happens to lie in $W^{1,2}$ with $|N|$ bounded below on an
annular neighborhood of an enclosing curve, then $N/|N|$ is there a
Sobolev map to the circle, and the degree theory of VMO maps
\cite{bn} supplies the winding along almost every nearby curve,
constant among them. What we do not know is the invariance under
quasiconformal changes of variables in either regime: the weight $\rho$
of \eqref{eq:rho-def} is then built from the a.e.-derivatives of $F$
and lies only in $L^2_{\mathrm{loc}}$, and our proof that $\rho$ cannot
wind spends continuity. One expects the conclusion to persist ---
$\rho$ is the derivative weight of a homeomorphism, and heuristically
carries no winding to spend --- but we record it as a conjecture:
\emph{the pseudo-analytic charge, where defined, is invariant under
quasiconformal changes of variables.}

\emph{Outlook.} The charge completes, at the topological end, the
program of \cite{mass, framed}: every invariant constructed there is an
integral of a density, and the present paper shows the same field $N$
also carries an integer. Two directions suggest themselves. Downward,
into the local theory: classify the behavior of solutions at a vortex
of local charge $n_p$, where the similarity principle governs the
solution but not the coefficient. Upward, out of the plane: in
quaternionic extensions of the theory the numerator field should become
a vector field on a domain in higher dimension, and the winding a
mapping degree into a sphere; whether the exactness of the substitution
law --- the positive factor $D$, on which everything here rests ---
survives the loss of commutativity is precisely the question.

\subsection*{Use of Generative AI Tools}
\medskip

The author discloses the use of Anthropic's Claude (Claude Fable 5,
accessed through the Claude.ai mobile interface in July 2026) in the
preparation of this manuscript. The tool was used as follows:

\begin{enumerate}
\item[(i)] \emph{Exploratory dialogue.} The organization of the paper
--- in particular the definition of the charge by enclosing curves, its
identification with the Brouwer degree
(Proposition~\ref{prop:degree}), the invariance of the vortex set
(Proposition~\ref{prop:vortex-set}), the exact/mod-$2$ dichotomy on
multiply connected domains (Remark~\ref{rem:multiply}), and the
independence construction of Theorem~\ref{thm:independence} --- emerged
in iterative research sessions, building on the companion papers
\cite{mass, framed}.

\item[(ii)] \emph{Drafting and revision of prose.} The manuscript was
drafted in iterative dialogue; all claims and their precise wording
were reviewed by the author.
\end{enumerate}

The author takes full responsibility for the correctness, accuracy,
originality, and integrity of all content.

\subsection*{Disclosure of interest}

The author reports there are no competing interests to declare.

\end{document}